%% file: main.tex
\documentclass{article}
\input{preamble.tex}

\title{The pro-Nisnevich topology}

\begin{document}

\maketitle

\begin{abstract}
    We construct the pro-Nisnevich topology, an analog of the pro-étale topology.
    We then show that the Nisnevich $\infty$-topos embeds into 
    the pro-Nisnevich $\infty$-topos, and that the pro-Nisnevich $\infty$-topos 
    is locally of homotopy dimension $0$. 
\end{abstract}

\section*{Introduction}
In \cite{bhatt2014proetale}, Bhatt and Scholze constructed the pro-étale topology,
and showed that the étale topos over a scheme embeds into the pro-étale topos.

In this short note, we do a similar construction for the Nisnevich topology.
Let $S$ be a qcqs scheme of finite Krull-dimension (e.g.\ a field), and write $\smooth{S}$ for the category of smooth quasi-compact $S$-schemes,
equipped with the Nisnevich topology \cite[Appendix A]{bachmann2020norms}. Write $\ShvNis{\smooth{S}}$ for the $\infty$-topos 
of sheaves on this site.
We will construct a site $(\prosmooth{S}, \pronis)$, called the \emph{pro-Nisnevich site}.
Write $\ShvProNisH{\prosmooth{S}}$ for the $\infty$-topos of hypersheaves on this site.
This $\infty$-topos is called the \emph{pro-Nisnevich topos}.
We prove the following results:
\begin{introthm}[\cref{lem:psig-equiv}]
    There is a full subcategory $W \subset \prosmooth{S}$
    and an equivalence $\ShvProNisH{\prosmooth{S}} \cong \PSig{W}$.
    In particular, the pro-Nisnevich topos is locally of homotopy dimension $0$.
\end{introthm}

\begin{introthm}[\cref{lem:embedding}]
    There is a geometric morphism of $\infty$-topoi 
    \begin{equation*}
        \nu^* \colon \ShvNis{\smooth{S}} \rightleftarrows \ShvProNisH{\prosmooth{S}} \colon \nu_*,
    \end{equation*}
    where $\nu_*$ is given by restriction and $\nu^*$ is fully faithful.
     
    Moreover, an $n$-truncated sheaf $F \in \ShvProNisH{\prosmooth{S}}$ is in the essential image of $\nu^*$ 
    if and only if for all $U \in \prosmooth{S}$ and all presentations of $U$ as a cofiltered limit
    $U \cong \limil{i} U_i$ (with the $U_i \in \smooth{S}$) 
    the canonical comparison map $\colimil{i} F(U_i) \to F(U)$ is an equivalence.
\end{introthm}

\subsection*{Acknowledgement}
I thank Tom Bachmann for helpful discussions, and Timo Weiß for reading 
a draft of this article.

The author acknowledges support by the Deutsche Forschungsgemeinschaft
(DFG, German Research Foundation) through the Collaborative Research
Centre TRR 326 \textit{Geometry and Arithmetic of Uniformized
Structures}, project number 444845124.

\section{The Pro-Nisnevich Topology}
In this section, we develop an analog of the pro-étale topology from \cite{bhatt2014proetale},
adapted for the Nisnevich topology.
We show that every affine scheme can be covered in this topology by a cdw-contractible ring,
an analog of w-contractible rings from \cite{bhatt2014proetale}.

Recall the following definition from \cite[Section 2.1 and 2.2]{bhatt2014proetale},
especially the alternative characterizations of w-local spaces from \cite[Lemma 2.1.4]{bhatt2014proetale}:
\begin{defn}
    Let $A$ be a ring.
    We say that $A$ is w-local 
    if the set of closed points $\Spec{A}^c$ is closed in $\Spec{A}$,
    and the map $\Spec{A}^c \to \Spec{A} \to \pi_0(\Spec A)$ is bijective.
\end{defn}

Recall the following definitions:
\begin{defn}
    Let $f \colon A \to B$ be a ring map.
    We say that
    \begin{enumerate}
        \item $f$ is a local isomorphism if for every prime ideal $\pideal{q} \subset B$ 
        there exists a $g \in B$, $g \notin \pideal{q}$ such that $A \to B_g$ induces an open immersion $\Spec{B_g} \to \Spec{A}$,
        \item $f$ is an \emph{ind-Zariski localization} 
        if $f$ can be written as a filtered colimit of local isomorphisms $f_i \colon A \to B_i$, and 
        \item $f$ is an \emph{ind-Zariski cover} if $f$ is a faithfully flat ind-Zariski localization.
    \end{enumerate}
\end{defn}

We have the following results about w-local rings:
\begin{lem} \label{lemma:w-local-cover-exists}
    Let $A$ be a ring.
    Then there exists an ind-Zariski cover $A \to A^Z$
    with $A^Z$ w-local, such that $\Spec{A^Z}^c \to \Spec{A}$
    is bijective.
\end{lem}
\begin{proof}
    Combine \cite[Lemma 2.2.4]{bhatt2014proetale} and \cite[Lemma 2.1.10]{bhatt2014proetale}.
\end{proof}

\begin{lem}\label{lemma:zariski-cover-of-profinite-map}
    For a ring $A$ and a map $T \to \pi_0(\Spec A)$ of totally disconnected compact Hausdorff spaces, 
    there is an ind-Zariski localization $A \to B$ such that $\Spec B \to \Spec A$ 
    gives rise to the given map $T \to \pi_0(\Spec A)$ on applying $\pi_0$.
\end{lem}
\begin{proof}
    This is \cite[Lemma 2.2.8]{bhatt2014proetale}.
\end{proof}

\subsection{Absolutely flat rings}

Recall the following definitions from \cite{bhatt2014proetale}:
\begin{defn} 
    Let $A$ be a ring. We say that $A$ is \emph{absolutely flat} if $A$
    is reduced and has dimension $0$. 

    Let $f \colon A \to B$ be a ring map.
    We say that $f$ is \emph{weakly étale} if $f$ is flat with 
    flat diagonal.
    We say that $f$ is \emph{ind-étale} if $f$ can be written as a 
    filtered colimit of étale morphisms $f_i \colon A \to B_i$.
\end{defn}

\begin{lem} \label{lemma:weakly-etale-absolutely-flat}
    Let $f \colon A \to B$ be a morphism of rings with $A$ absolutely flat.
    Suppose that $f$ is weakly étale. Then $B$ is absolutely flat. 
\end{lem}
\begin{proof}
    See \cite[\href{https://stacks.math.columbia.edu/tag/092I}{Tag 092I}]{stacks-project}.
\end{proof}

\begin{lem} \label{lemma:absolutely-flat-profinite-set}
    Let $A$ be a ring of dimension $0$ (e.g.\ $A$ absolutely flat).
    Then $\Spec{A}$ is totally disconnected compact Hausdorff.
\end{lem}
\begin{proof}
    Clearly $\Spec{A}$ is compact.
    It is Hausdorff, see e.g. \cite[\href{https://stacks.math.columbia.edu/tag/0CKV}{Tag 0CKV}]{stacks-project},
    which applies since $\Spec{A}$ is affine and hence separated.
    Since $\Spec{A}$ is spectral, it has a basis of compact open subsets.
    In Hausdorff spaces, compact subsets are closed,
    thus $\Spec{A}$ has a basis of clopen subsets.
    This implies that $A$ is totally disconnected.
\end{proof}

\begin{cor}
    Let $A$ be absolutely flat. Then $A$ is w-local.
\end{cor}
\begin{proof}
    We know that $\Spec{A}$ is totally disconnected compact Hausdorff from \cref{lemma:absolutely-flat-profinite-set}.
    In particular, all points of $\Spec{A}$ are closed
    and $\Spec{A} \cong \pi_0(\Spec A)$.
\end{proof}

\begin{lem} \label{lemma:colimit-of-absolutely-flat}
    Let $A = \colimil{i} A_i$ be a filtered colimit 
    of absolutely flat rings.
    Then $A$ is absolutely flat.
\end{lem}
\begin{proof}
    Being reduced and of dimension $0$ are properties of rings
    that are stable under filtered colimits:
    For the first claim, note that a nilpotent element in a ring 
    is the same as a ring map from $\Z[X]/(X^n)$ for some $n \ge 2$.
    But since $\Z[X]/(X^n)$ is of finite presentation over $\Z$,
    we see that $\Hom{}(\Z[X]/(X^n), \colimil{i} A_i) \cong \colimil{i} \Hom{}(\Z[X]/(X^n), A_i) = 0$,
    since the $A_i$ are reduced.
    For the second claim, 
    use \cite[\href{https://stacks.math.columbia.edu/tag/01YW}{Tag 01YW} and \href{https://stacks.math.columbia.edu/tag/01YY}{Tag 01YY (2)}]{stacks-project}.
\end{proof}

\begin{defn} 
    Let $f \colon \Spec B \to \Spec A$ be a map of affine schemes.
    We define the \emph{completely decomposed locus} of $f$ as the subset 
    \begin{equation*}
        \cd{f} \coloneqq \set{x \in \Spec B}{k(f(x)) \to k(x) \text{ is an isomorphism}} \subset \Spec{B}.
    \end{equation*}
    We say that $f$ is \emph{completely decomposed} if $f(\cd{f}) = \Spec A$,
    i.e.\ over every point of $\Spec A$ there is a point in $\Spec B$ that induces
    an isomorphism on residue fields.
\end{defn}

\begin{lem} \label{lemma:completely-decomposed-locus-of-ind-map}
    Let $A = \colimil{i \in I} A_i$ be a filtered colimit of rings.
    Suppose wlog $I$ has an initial element $0$.
    Denote by $f_i \colon \Spec{A_i} \to \Spec{A_0}$ the map induced by $0 \to i$,
    and by $f_\infty \colon \Spec A \to \Spec{A_0}$ the projection.

    Then $\cd{f_\infty} = \limil{i} \, \cd{f_i}$.
\end{lem}
\begin{proof}
    Let $p_i \colon \Spec A \to \Spec{A_i}$ be the projections.
    Let $x_\infty \in \Spec{A}$,
    and $x_i \coloneqq p_i(x_\infty) \in \Spec{A_i}$.
    
    Suppose $x_\infty \in \cd{f_\infty}$ and thus $k(x_0) \cong k(x_\infty)$.
    Thus, for all $i$ we have a factorization of this isomorphism $k(x_0) \to k(x_i) \to k(x_\infty)$.
    Hence $k(x_i) \cong k(x_0)$ and we get $x_i \in \cd{f_i}$.
    This implies $x_\infty \in \limil{i} \, \cd{f_i}$.
    
    On the other hand, suppose that $x_\infty \in \limil{i} \, \cd{f_i}$.
    This implies that $x_i \in \cd{f_i}$.
    By \cite[\href{https://stacks.math.columbia.edu/tag/0CUG}{Lemma 0CUG}]{stacks-project},
    we know that $k(x_\infty) = \colimil{i} k(x_i)$.
    But $k(x_0) \cong k(x_i)$ and thus $k(x_\infty) \cong \colimil{i} k(x_0) \cong k(x_0)$.
    Hence $x_\infty \in \cd{f_\infty}$.
\end{proof}

\begin{lem} \label{lem:component-limit}
    Let $A$ an absolutely flat ring.
    Write $A \cong \colimil{i} A_i$ as a filtered colimit of its finitely generated subrings.
    Let $x \in \Spec{A}$.
    For each $i$ write $p_i \colon \Spec{A} \to \Spec{A_i}$ for the projection,
    and $Z_i(x) \subset \Spec{A_i}$ for the connected component of $p_i(x)$.
    Then $\{x\} \cong \limil{i} Z_i(x)$.
\end{lem}
\begin{proof}
    Clearly $x \in \limil{i} Z_i(x)$.
    Suppose on the other hand that $y \in \Spec{A}$, with $y \neq x$.
    Since $\Spec{A}$ is totally disconnected compact Hausdorff by \cref{lemma:absolutely-flat-profinite-set},
    there exists a clopen subset $U \subset \Spec{A}$ into clopen subsets with $x \in U$ and $y \notin U$.
    By descent, there is an index $j$, and $U_j \subset \Spec{A_j}$ clopen 
    such that $U = \Spec{A} \times_{\Spec{A_j}} U_j$.
    But then $Z_i(x) \subseteq \Spec{A_i} \times_{\Spec{A_j}} U_j$ for all $i \ge j$,
    and hence $\limil{i} Z_i(x) \subseteq U$.
    Since $y \notin U$, we conclude $y \notin \limil{i} Z_i(x)$.    
\end{proof}

\begin{lem} \label{lemma:completely-decomposed-locus-of-etale-map}
    Let $f \colon A \to B$ be an étale ring map, with $A$ absolutely flat.
    Then $\cd{f} \subset \Spec{B}$ is closed.
\end{lem}
\begin{proof}
    Write $A = \colimil{i \in I} A_i$ as a filtered colimit of its finitely generated subrings.
    As $f$ is of finite presentation, by descent there is an index $0 \in I$, a ring $B_0$ and an étale morphism $f_0 \colon B_0 \to A_0$
    such that $f$ is the basechange of $f_0$.
    Write $B_i \coloneqq B_0 \otimes_{A_0} A_i$ for each $i \ge 0$,
    and $f_i \colon A_i \to B_i$.
    In particular, $\Spec{B} \cong \limil{i \ge 0} \Spec{B_i}$.
    Note that $\Spec{B_i}$ is of finite type over $\Spec{\Z}$ (as it is étale over $\Spec{A_i}$ which is of finite type over $\Spec{\Z}$).
    In particular, the connected components of $\Spec{B_i}$ are open.
    Write $W_i \subset \Spec{B_i}$ for the union of those components $W_{i,j} \subset \Spec{B_i}$
    such that $f_i|_{W_{i,j}} \colon W_{i,j} \to \Spec{A_i}$ is an isomorphism onto 
    a component of $\Spec{A_i}$.
    It suffices to prove that $\cd{f} = \limil{i} W_i$:
    Indeed, since $\Spec{B_i}$ has only finitely many components,
    we see that $W_i$ is closed.
    Thus, also $\cd{f} = \limil{i} W_i$ is closed as a limit of closed subsets.

    It is clear that $\limil{i} W_i \subseteq \cd{f}$,
    as $W_i \subseteq \cd{f_i}$.
    So let $x \in \cd{f}$ and $y \coloneqq f(x)$.
    We will write $p_i \colon \Spec{A} \to \Spec{A_i}$ for the projection,
    and $Z_i(y) \subset \Spec{A_i}$ for the component of $p_i(y)$.
    Let $V_i(y) \coloneqq \Spec{B_i} \times_{\Spec{A_i}} Z_i(y) \subseteq \Spec{B_i}$,
    which is a clopen subset.
    As the inclusion $\{x\} \hookleftarrow f^{-1}(y)$ is clopen (because $f^{-1}(y)$ 
    is a discrete set of points), we may assume by descent (possibly by replacing $0$
    by a larger index), that there exists a $T_0 \subseteq V_0(y)$ clopen 
    such that $\{x\} = T_0 \times_{V_0(y)} f^{-1}(y)$,
    and set $T_i \coloneqq \Spec{B_i} \times_{\Spec{B_0}} T_0$ (which is clopen in 
    $V_i(y)$).
    Note that $T_i \times_{Z_i(y)} \{y\} \cong T_i \times_{V_i(y)} f^{-1}(y) \cong \{x\}$.
    As $f$ is completely decomposed at $x$ by definition, it induces an isomorphism 
    $\{x\} \xrightarrow{\simeq} \{y\}$.
    Thus, again by descent and \cref{lem:component-limit}, we conclude
    that $T_i \xrightarrow{\simeq} Z_i(y)$ for $i \gg 0$.
    But $T_i \subseteq V_i(y)$ is clopen, and hence $T_i \subseteq V_i(y) \subseteq \Spec{B_i}$
    is also clopen, and maps isomorphically to $Z_i(y)$. 
    Therefore, $T_i \subseteq W_i$ and we conclude that $x \in \limil{i} W_i$.
\end{proof}

\begin{cor} \label{lemma:completely-decomposed-locus-of-ind-etale-map}
    Let $A$ be absolutely flat and $f \colon A \to B$ an ind-étale algebra.
    Then $\cd{f}$ is closed.
\end{cor}
\begin{proof}
    As a limit of closed subsets is closed, 
    we can first reduce to the case that $f$ is étale via \cref{lemma:completely-decomposed-locus-of-ind-map},
    and then conclude by \cref{lemma:completely-decomposed-locus-of-etale-map}.
\end{proof}

\subsection{Weakly Contractible Covers}

\begin{defn} 
    Let $f \colon A \to B$ be a ring map.
    We say that $f$ is a \emph{Nisnevich cover} if $f$ is étale and completely decomposed;
    and that $f$ is an \emph{ind-Nisnevich cover} if $f$ is ind-étale and 
    completely decomposed.
\end{defn}

\begin{defn} 
    A ring $A$ is called \emph{cdw-contractible} if every ind-Nisnevich-cover $A \to B$
    splits.
\end{defn}

\begin{defn}
    Let $A$ be a ring.
    We write $\ind{et}{A}$ for the category of ind-étale $A$-algebras.
\end{defn}

Recall the following definition from \cite[Definition 2.2.10]{bhatt2014proetale}:
\begin{defn} 
    Let $A \to B$ be a ring map.
    We have a base-change functor $\ind{et}{A} \to \ind{et}{B}$.
    Define $\hens{A}{-} \colon \ind{et}{B} \to \ind{et}{A}$
    via \[B' \mapsto \colim{A \xrightarrow{et} A' \to B'} A'.\]
    This is a right adjoint to the base change functor
    and called \emph{henselization}.

    We say that a ring $A$ is \emph{henselian along an ideal $I$}
    if $\hens{A}{A/I} \cong A$.
\end{defn}

\begin{lem} \label{lemma:henselization-henselian-along-ideal-with-right-quotient}
    Let $A$ be a ring, and $I \subseteq A$ an ideal contained in
    the Jacobson radical $I_A$ of $A$.
    Let $\overline{A}$ be an ind-étale $A/I$-algebra.
    Then $B \coloneqq \hens{A}{\overline{A}}$ is henselian along $IB$
    with $B/IB \cong \overline{A}$.
\end{lem}
\begin{proof}
    We first prove the second part.
    Since $\overline{A}$ is an ind-étale $A/I$-algebra, we know that $\hens{A/I}{\overline A} \cong \overline{A}$.
    Note that $B/IB = \colimil{A \xrightarrow{et} A' \to \overline{A}} A'/IA'$.
    By basechange, $A/I \to A'/IA'$ is an étale morphism for all such $A'$.
    Thus, it suffices to show that every factorization $A/I \xrightarrow{et} C \to \overline{A}$
    is the basechange of a factorization $A \to A' \to \overline{A}$.
    Note that $A/I \to C$ can be lifted to an étale morphism $A \to A'$ by \cite[\href{https://stacks.math.columbia.edu/tag/04D1}{Tag 04D1}]{stacks-project}.
    But since $A' /IA' = C$, this fits into a factorization $A \to A' \to \overline{A}$ as desired. 

    For the first part we need to show that $B \xrightarrow{\operatorname{id}} B$ is final under étale morphisms $B \to B'$ such that 
    there is a factorization $B \to B' \to B/IB = \overline{A}$.
    So let $B \to B'$ such an étale morphism.
    We need to show that it has a section.
    But $B = \colimil{A \xrightarrow{et} A' \to \overline{A}} A'$.
    By finite presentation of $B \to B'$ there is thus a factorization $A \to A' \to \overline{A}$
    with $A \to A'$ étale and an étale morphism $A' \to \tilde{B}$
    such that $B \otimes_{A'} \tilde{B} \cong B'$.
    Since $A' \to \tilde{B}$ is étale, also $A \to \tilde{B}$ is étale.
    Note that there is a factorization $A \to \tilde{B} \to \overline{A}$.
    We then get a canonical map $\tilde{B} \to B$ by definition of $B$.
    But then $B' \cong B \otimes_{A'} \tilde{B} \to B \otimes_{A'} B \to B$
    gives the desired section.
\end{proof}

\begin{lem} \label{lemma:cdw-contractible-if-cdw-contractible-mod-henselian-ideal}
    Let $A$ be a ring which is henselian along an ideal $I$.
    Suppose $A/I$ is cdw-contractible.
    Then $A$ is cdw-contractible.
\end{lem}
\begin{proof}
    Let $f \colon A \to B$ be an ind-Nisnevich cover.
    Then $f/I \colon A/I \to B/IB$ is an ind-Nisnevich cover and thus has a section.
    But 
    \begin{equation*}
        \Hom{A}(B, A) \cong \Hom{A}(B, \hens{A}{A/I}) \cong \Hom{A/I}(B/IB, A/I).
    \end{equation*}
    This gives us the desired section of $f$.
\end{proof}

\begin{defn}
    Let $T$ be a totally disconnected compact Hausdorff space.
    We say that $T$ is \emph{extremally disconnected} 
    if every surjection $S \twoheadrightarrow T$ from a 
    totally disconnected compact Hausdorff space $S$ has a section.
\end{defn}

\begin{lem} \label{lemma:pi0-extremally-disconnected-iff-cdw-contractible}
    Let $A$ be absolutely flat.
    Then $A$ is cdw-contractible if and only if 
    $\Spec{A} = \pi_0(\Spec A)$ is extremally disconnected.
\end{lem}
\begin{proof}
    First we see from \cref{lemma:absolutely-flat-profinite-set} that because $A$ is absolutely flat, the topological space $\Spec{A}$
    is a totally disconnected compact Hausdorff space, so in particular $\Spec{A} \cong \pi_0(\Spec{A})$.

    Suppose that $A$ is cdw-contractible.
    Let $f_0 \colon T \to \Spec A$ be a surjective morphism of totally disconnected compact Hausdorff spaces.
    But $f_0$ is induced by an ind-Zariski cover $A \to B$, see \cref{lemma:zariski-cover-of-profinite-map},
    which has a section because $A$ is cdw-contractible and ind-Zariski covers 
    are ind-Nisnevich covers. This gives us the topological section $\Spec{A} \cong \pi_0(\Spec A) \to \pi_0(\Spec B) \cong T$.

    Suppose now that $\Spec A$ is extremally disconnected.
    Let $f \colon A \to B$ be an ind-Nisnevich cover.
    Thus, $B$ is absolutely flat, see \cref{lemma:weakly-etale-absolutely-flat} (as ind-étale maps are weakly étale, see \cite[\href{https://stacks.math.columbia.edu/tag/092N}{Tag 092N}]{stacks-project}).
    Using \cref{lemma:completely-decomposed-locus-of-ind-etale-map}, we see that $\cd{\Spec{f}} \subseteq \Spec{B}$
    is closed. Since $f$ is completely decomposed,
    we see that $\cd{\Spec{f}} \to \Spec{B} \to \Spec{A}$ is in fact surjective.
    Let $B'$ be the $B$-algebra realizing the closed subset $\cd{\Spec f}$
    with its reduced scheme structure.
    Then $B'$ is reduced and thus absolutely flat.
    By definition of $B'$, $h \colon \Spec{B'} \to \Spec{A}$ induces isomorphisms 
    on all residue fields. 
    Now, $h$ has a topological section $s$ because $\Spec{A} = \pi_0(\Spec A)$ is extremally disconnected,
    and $\Spec{B'}$ is totally disconnected compact Hausdorff by \cref{lemma:absolutely-flat-profinite-set}.
    Let $\Spec{B''}$ be the reduced closed subscheme of $\Spec{B'}$ realizing the image $s(\Spec A) \subset \Spec{B'}$
    (note that the image of $s$ is closed because $s$ is a map between compact Hausdorff spaces).
    This maps isomorphically to $\Spec{A}$, which gives us the section
    \begin{equation*}
        \Spec{A} \cong \Spec{B''} \subseteq \Spec{B'} \subseteq \Spec{B}.
    \end{equation*}
\end{proof}

\begin{lem} \label{lem:contractible-cover}
    Let $A$ be a ring. Then there is an ind-Nisnevich cover $A \to \overline A$ with
    $\overline A$ a cdw-contractible ring.
\end{lem}
\begin{proof}
    Consider the ind-Zariski cover $A \to A^Z$ by a w-local ring from \cref{lemma:w-local-cover-exists}.
    Note that $A^Z/I_{A^Z}$ is absolutely flat, using \cite[Lemma 2.2.3]{bhatt2014proetale}.
    Let $T$ be an extremally disconnected compact Hausdorff space covering $\Spec{A^Z/I_{A^Z}}$.
    Let $A'$ be an ind-Zariski cover of $A^Z/I_{A^Z}$ such that $\Spec{A'} \to \Spec{A^Z/I_{A^Z}}$
    realizes the map $T \to \Spec{A^Z/I_{A^Z}}$, see \cref{lemma:zariski-cover-of-profinite-map}.
    Then $A'$ is absolutely flat, and cdw-contractible by \cref{lemma:pi0-extremally-disconnected-iff-cdw-contractible}.
    Define $\overline{A} \coloneqq \hens{A^Z}{A'}$.
    Then $\overline{A}$ is henselian along $I_{A^Z}\overline{A}$,
    with quotient $\overline{A}/I_{A^Z}\overline{A} \cong A'$ (see \cref{lemma:henselization-henselian-along-ideal-with-right-quotient}),
    hence $\overline{A}$ is cdw-contractible using \cref{lemma:cdw-contractible-if-cdw-contractible-mod-henselian-ideal}.
    Now note that $A^Z \to \overline{A}$ induces isomorphisms of residue 
    fields at the closed points of $\Spec{A^Z}$.
    Note that $\Spec{A^Z}^c \to \Spec{A}$ is bijective (see \cref{lemma:w-local-cover-exists}) and induces 
    isomorphisms of residue fields.
    Thus, $\Spec{\overline{A}} \to \Spec{A}$ is completely decomposed.
    Since $A \to A^Z$ is ind-Zariski, and $A^Z \to \overline{A}$
    is ind-étale by construction, we conclude that $A \to \overline{A}$
    is an ind-Nisnevich cover.
\end{proof}

\section{The Pro-Nisnevich \texorpdfstring{$\infty$}{infinity}-Topos}
Let $S$ be a qcqs scheme with finite Krull-dimension.
\begin{defn}
    Write $\prosmooth{S}$ for the category of pro-smooth schemes over $S$,
    i.e. morphisms $X \to S$ such that $X$ can be written as a cofiltered limit $X \cong \limil{i} X_i$
    where $X_i \in \smooth{S}$.
    Write $\prosmoothaff{S}$ for the full subcategory of pro-smooth schemes over $S$ which are affine.
\end{defn}

\begin{defn}
    Let $\mathcal U \coloneqq \{f_i \colon U_i \to U\}_{i \in I}$ be a family of morphisms in $\prosmooth{S}$.
    We say that $\mathcal U$ is a \emph{pro-Nisnevich cover} if 
    \begin{itemize}
        \item $f_i$ is pro-étale for all $i \in I$,
        \item the morphism $\amalg_i U_i \to U$ is completely decomposed, and 
        \item the $f_i$ form an fpqc-cover \cite[\href{https://stacks.math.columbia.edu/tag/03NW}{Tag 03NW}]{stacks-project}.
    \end{itemize}

    Similarly, let $\mathcal U \coloneqq \{f_i \colon U_i \to U\}_{i \in I}$ be a family of morphisms in $\prosmoothaff{S}$.
    We say that $\mathcal U$ is a \emph{pro-Nisnevich cover}
    if it is a pro-Nisnevich cover in $\prosmooth{S}$.
\end{defn}

\begin{rmk}
    Let $\Spec{f} \colon \Spec{B} \to \Spec{A}$ be a morphism in $\prosmoothaff{S}$. 
    Then $\{\Spec{f}\}$ is a pro-Nisnevich cover if and only if $f \colon A \to B$
    is an ind-Nisnevich cover.
\end{rmk}

Recall the definition of a site \cite[\href{https://stacks.math.columbia.edu/tag/00VH}{Tag 00VH}]{stacks-project}
and a morphism of sites \cite[\href{https://stacks.math.columbia.edu/tag/00X1}{Tag 00X1}]{stacks-project}.
\begin{lem} \label{lem:sites}
    There are sites $(\prosmooth{S}, \pronis)$ and $(\prosmoothaff{S}, \pronis)$,
    where the covers are given by pro-Nisnevich covers.
\end{lem}
\begin{proof}
    The only nontrivial part is the stability of covers under pullbacks.
    First note that the pullback of a pro-étale cover exists and is again 
    a pro-étale cover, which follows from \cite[\href{https://stacks.math.columbia.edu/tag/098L}{Lemma 098L}]{stacks-project}.
    Since completely decomposed families of morphisms are stable under pullbacks in the category of all $S$-schemes,
    the result follows as pullbacks in $\prosmooth{S}$ along pro-étale morphisms 
    coincide with the underlying pullback in the category of all $S$-schemes 
    (this can be proven analogously to \cite[\href{https://stacks.math.columbia.edu/tag/098M}{Tag 098M}]{stacks-project}).
\end{proof}

We can define the following $\infty$-topoi:
\begin{defn}
    We will write 
    \begin{itemize}
        \item $\ShvProNisNH{\prosmooth{S}}$ for the $\infty$-topos of sheaves of anima on the site $(\prosmooth{S}, \pronis)$,
        \item $\ShvProNisH{\prosmooth{S}}$ for the hypercompletion of $\ShvProNisNH{\prosmooth{S}}$,
        \item $\ShvProNisNH{\prosmoothaff{S}}$ for the $\infty$-topos of sheaves of anima on the site $(\prosmoothaff{S}, \pronis)$, and
        \item $\ShvProNisH{\prosmoothaff{S}}$ for the hypercompletion of $\ShvProNisNH{\prosmoothaff{S}}$.
    \end{itemize}
\end{defn}

\begin{lem} \label{lem:affine-non-affine-comparison}
    There are equivalences of $\infty$-topoi
    \begin{equation*}
        \ShvProNisNH{\prosmoothaff{S}} \cong \ShvProNisNH{\prosmooth{S}}
    \end{equation*}
    and 
    \begin{equation*}
        \ShvProNisH{\prosmoothaff{S}} \cong \ShvProNisH{\prosmooth{S}}
    \end{equation*} 
    induced by left Kan extension along the inclusion $\prosmoothaff{S} \hookrightarrow \prosmooth{S}$.
\end{lem}
\begin{proof}
    Note that the inclusion $u \colon \prosmoothaff{S} \to \prosmooth{S}$ preserves covers by definition,
    and commutes with limits as limits of affine schemes are affine.
    Thus, the first equivalence is an easy application of \cite[Lemma C.3]{hoyois2015quadratic}.
    The second equivalence is just the hypercompletion of the first.
\end{proof}

\begin{defn}
    We write $W$ for the full subcategory of $\prosmoothaff{S}$
    consisting of (spectra of) cdw-contractible rings.
\end{defn}

\begin{lem} \label{lem:psig-topos}
    The category $W$ is extensive (\cite[Definition 2.3]{bachmann2020norms}),
    and the category $\PSig{W}$ is an $\infty$-topos.
\end{lem}
\begin{proof}
    The category is extensive, as it is a full subcategory of the category of schemes (which is extensive),
    stable under finite coproducts and summands.
    The second part is \cite[Lemma 2.4]{bachmann2020norms}.
\end{proof}

Recall the notion of a locally weakly contractible site \cite[Definition B.3]{mattis2024unstable}.
\begin{lem} \label{lem:locally-weakly-contractible}
    The site $(\prosmoothaff{S}, \pronis)$ is locally weakly contractible. 
\end{lem}
\begin{proof}
    We have seen that the category $W$ is extensive, see \cref{lem:psig-topos},
    and consists by definition of weakly contractible objects.
    The pro-Nisnevich topology is a $\Sigma$-topology, because clopen immersions 
    are in particular pro-étale and induce isomorphisms on the residue fields.
    The pro-Nisnevich topology on $\prosmooth{S}$
    is finitary (cf. \cite[Definition A.3.1.1]{sag}) by definition, so every object is quasicompact.
    Note that every object in $\prosmoothaff{S}$ has a cover by an object in $W$,
    this is the content of \cref{lem:contractible-cover}.
    This proves the lemma.
\end{proof}

\begin{thm} \label{lem:psig-equiv}
    There is an equivalence of $\infty$-topoi
    \begin{equation*}
        \ShvProNisH{\prosmooth{S}} \cong \PSig{W}.
    \end{equation*}
    Thus, $\ShvProNisH{\prosmooth{S}}$ is locally of homotopy dimension $0$ 
    and in particular Postnikov-complete.
\end{thm}
\begin{proof}
    There is an equivalence $\ShvProNisH{\prosmooth{S}} \cong \ShvProNisH{\prosmoothaff{S}}$,
    see \cref{lem:affine-non-affine-comparison}.
    Thus, the theorem follows from an application of \cite[Lemma B.7]{mattis2024unstable},
    because the site $(\prosmoothaff{S}, \pronis)$ is locally weakly contractible,
    see \cref{lem:locally-weakly-contractible}.
    See \cite[Lemma 4.20]{mattis2024unstable} for a proof that $\PSig{W}$ is locally of homotopy dimension $0$.
    The last claim is \cite[Proposition 7.2.1.10]{highertopoi}.
\end{proof}

\begin{lem} \label{lem:nis-post-complete}
    The $\infty$-topos $\ShvNis{\smooth{S}}$ is Postnikov-complete.
\end{lem}
\begin{proof}
    If $S$ is noetherian, then the proof of this result is completely analogous to the proof of \cite[Lemma 5.1]{mattis2024unstable}.
    Here we use that $S$ is qcqs of finite Krull dimension (and that schemes in $\smooth{S}$ possess the same property, note that for us,
    they are by definition quasi-compact).
    If $S$ is not noetherian, we can still apply the same strategy, but resort to \cite[Theorem 3.18]{clausen2021hyperdescent}
    instead of \cite[Theorem 3.7.7.1]{sag}.
\end{proof}

\begin{thm} \label{lem:embedding}
    There is a geometric morphism 
    \begin{equation*}
        \nu^* \colon \ShvNis{\smooth{S}} \rightleftarrows \ShvProNisH{\prosmooth{S}} \colon \nu_*,
    \end{equation*}
    where the right adjoint is given by restriction, and the left adjoint is fully faithful.

    Moreover, an $n$-truncated sheaf $F \in \ShvProNisH{\prosmooth{S}}$ is in the essential image of $\nu^*$ 
    if and only if for all $U \in \prosmooth{S}$ and all presentations of $U$ as cofiltered limit
    $U \cong \limil{i} U_i$ (with the $U_i \in \smooth{S}$) 
    the canonical map $\colimil{i} F(U_i) \to F(U)$ is an equivalence.
\end{thm}
\begin{proof}
    Both involved $\infty$-topoi are Postnikov-complete, see \cref{lem:psig-equiv} and \cref{lem:nis-post-complete}.
    We want to apply \cite[Proposition B.8]{mattis2024unstable}
    to the natural inclusion of sites 
    \begin{equation*}
        (\smooth{S}, \nis) \subset (\prosmooth{S},\pronis).
    \end{equation*}
    In the notation of \cite[Proposition B.8]{mattis2024unstable},
    it remains to prove that $\iota_h j^* F \cong k^* \iota_h' F$ for every $n$-truncated Nisnevich sheaf $F \in \ShvNis{\smooth{S}}$,
    i.e.\ we have to show that the presheaf $k^* \iota_h' F$ is already a pro-Nisnevich hypersheaf.
    The proof of this fact is completely analogous to the proof of \cite[Theorem B.24]{mattis2024unstable}.
\end{proof}

\bibliography{bibliography}
\end{document}

%% file: preamble.tex
\usepackage[utf8]{inputenc}
\usepackage[english]{babel}

\usepackage{amsmath, amsthm, amssymb, amsfonts}
\usepackage{mathtools}

\usepackage{tikz-cd}

\usepackage{hyperref}
\usepackage[capitalize,noabbrev]{cleveref}

\usepackage{enumitem}

\theoremstyle{definition}
\newtheorem{defn}{Definition}[section]

\theoremstyle{plain}
\newtheorem{lem}[defn]{Lemma}

\theoremstyle{plain}
\newtheorem{cor}[defn]{Corollary}

\theoremstyle{plain}

\theoremstyle{plain}
\newtheorem{thm}[defn]{Theorem}

\theoremstyle{plain}
\newtheorem{introthm}{Theorem}

\theoremstyle{plain}

\theoremstyle{remark}
\newtheorem{rmk}[defn]{Remark}

\crefname{lem}{Lemma}{Lemmas}
\crefname{defn}{Definition}{Definitions}
\crefname{cor}{Corollary}{Corollaries}
\crefname{prop}{Proposition}{Propositions}
\crefname{thm}{Theorem}{Theorems}
\crefname{rmk}{Remark}{Remarks}
\crefname{section}{Section}{Sections}

\newcommand{\Z}{{\mathbb{Z}}}

\newcommand{\Spec}[1]{\operatorname{Spec}(#1)}

\newcommand{\PSig}[1]{{\mathcal{P}_{\Sigma}(#1)}}

\newcommand{\colim}[1]{{\underset{\substack{#1}}{\operatorname{colim}}\,}}
\newcommand{\colimil}[1]{{{\operatorname{colim}}_{#1}\,}}
\renewcommand{\lim}[1]{{\underset{\substack{#1}}{\operatorname{lim}}\,}}
\newcommand{\limil}[1]{{{\operatorname{lim}}_{#1}\,}}

\newcommand{\Hom}[1]{{\operatorname{Hom}_{#1}}}

\newcommand{\set}[2]{\left\{\, {#1} \,\middle\vert\, {#2} \,\right\}}

\newcommand{\ShvTopH}[2]{\operatorname{Shv}^{\operatorname{h}}_{#1}({#2})}
\newcommand{\ShvTopNH}[2]{\operatorname{Shv}^{\operatorname{nh}}_{#1}({#2})}

\newcommand{\pronis}{\operatorname{pronis}}

\newcommand{\ShvNis}[1]{\operatorname{Shv}_{\operatorname{nis}}(#1)}
\newcommand{\ShvProNisH}[1]{\ShvTopH{\pronis}{#1}}
\newcommand{\ShvProNisNH}[1]{\ShvTopNH{\pronis}{#1}}

\newcommand{\cd}[1]{{\operatorname{cd}({#1})}}

\newcommand{\pideal}[1]{{\mathfrak #1}}
\newcommand{\hens}[2]{{\operatorname{Hens}_{#1}(#2)}}

\newcommand{\smooth}[1]{\operatorname{Sm}_{#1}} 
\newcommand{\prosmooth}[1]{\operatorname{ProSm}_{#1}} 
\newcommand{\prosmoothaff}[1]{\operatorname{ProSmAff}_{#1}} 

\newcommand{\ind}[2]{\operatorname{Ind}_{#1}(#2)}

\newcommand{\nis}{\operatorname{nis}}

\author{Klaus Mattis\footnote{\href{mailto:klaus.mattis@uni-mainz.de}{klaus.mattis@uni-mainz.de}}}
\date{\today}